\documentclass[11pt]{article}
\usepackage{amsmath, amssymb}
\usepackage{mathtools}
\usepackage{algorithm}
\usepackage{algpseudocode}
\usepackage{fmtcount}
\usepackage{graphicx}
\usepackage{enumitem} 
\usepackage{subcaption}
\usepackage{afterpage}
\usepackage[noadjust]{cite}
\usepackage[font=small,labelfont=bf]{caption}
\usepackage[labelfont=bf]{caption}
\usepackage{setspace}
\allowdisplaybreaks
\usepackage[title,titletoc]{appendix}
\usepackage{epsfig}
\usepackage{graphicx}
\usepackage{hhline}
\usepackage{latexsym}
\usepackage{amssymb}
\usepackage{amsmath,amscd}
\usepackage{multirow} 
\usepackage{array}
\usepackage{caption}
\usepackage{subcaption}
\usepackage{authblk}
\usepackage{color}
\usepackage{natbib}
\usepackage{rotating}
\usepackage{graphicx,lscape}
\usepackage{amsmath, amsthm, amssymb}
\usepackage{graphicx}
\usepackage{epsfig}
\usepackage{graphics,color, graphicx}
\usepackage{latexsym}
\usepackage{fullpage}
\usepackage{booktabs,fixltx2e}
\usepackage[flushleft]{threeparttable}
\usepackage{amssymb,amsmath,amscd}
\usepackage{fancyvrb}
\usepackage{pdfpages}
\usepackage{fmtcount}
\usepackage{url}
\usepackage[margin=1in]{geometry}

\theoremstyle{plain}
\newtheorem{proposition}{Proposition}[section]

\theoremstyle{definition}

\newtheorem{assumption}{Assumption}[section]
\usepackage{setspace}

\title{\bf Asymptotic theory for multiple samples with flexible random membership
	\medskip
}

\author[1,2]{Ha-Young Shin}
\affil[1]{
	Department of Statistics and Actuarial Science, Soongsil University, 369 Sangdo-ro, Dongjak-gu 06978, Seoul, Korea
}
\affil[2]{
	Integrative Institute of Basic Sciences, Soongsil University, 369 Sangdo-ro, Dongjak-gu 06978, Seoul, Korea
}
{
    \makeatletter
    \renewcommand\AB@affilsepx{: \protect\Affilfont}
    \makeatother

    \affil[ ]{Email}

    \makeatletter
    \renewcommand\AB@affilsepx{, \protect\Affilfont}
    \makeatother

    \affil[1,2]{hayoung.shin@gmail.com}
}
\date{}

\begin{document}
	\maketitle
	
	\begin{abstract}

		\noindent
		A statistic can be a function of multiple samples. There is little existing work on asymptotic theory for such statistics when group membership is neither fixed nor iid. We propose a flexible framework that can handle different kinds of deterministic and random membership, and mixtures of both. We prove some asymptotic properties and apply the framework to the stratified sampling context.
		\vspace{\baselineskip}
		
		
	\end{abstract}
	
	\pagenumbering{arabic}
	
	\section{Introduction} \label{intro}

    We are often faced with statistics based on multiple samples, such as pooled variances or two-sample test statistics. Most existing asymptotic theory with multiple samples mostly treats group membership as deterministic e.g. \cite{Vardi1985, Rublik2007, Aoshima2015, Vaillancourt1995}. Despite the existence of techniques like post-stratification, it is surprisingly difficult to find asymptotic work for multiple samples with non-iid random membership. Random membership arises naturally. Consider these scenarios for a study on a demographically stratified population:
\begin{enumerate}
    \item Researchers determine beforehand how many people to sample from each of the demographic groups.
    \item Researchers take an iid sample from the population, which is then stratified into the demographic groups.
    \item A certain demographic group is known to be reluctant to participate in the study, for various sociological reasons. Thus the researchers provide extra incentives for members of this demographic group to volunteer.
    \item Researchers want to ensure that the sample size for each demographic group exceeds some minimum threshold.
\end{enumerate}
    
    We propose a flexible framework for asymptotic theory with multiple samples in which membership is directly modeled as a random variable. This allows for random, deterministic, and partially deterministic group membership, covering all of the above scenarios and many more. %Section \ref{framework} describes the framework and proves some basic properties. Section \ref{stratified} shows how it can be applied to a stratified sampling setting.
    All proofs are in Appendix \ref{proofs}.

\section{Framework} \label{framework}

All random elements in this paper are defined on the probability space $(\Omega,\mathcal{F},P)$, unless stated otherwise; for notational simplicity, we will suppress the mention of $\omega\in\Omega$ for all maps and sets unless needed. $I\{\cdot\}$ denotes an indicator function. The measurable space $M$ has $\sigma$-algebra $\mathcal{B}$. The spaces $\Xi:=\{1,\ldots,\xi\}$ ($\xi\in\mathbb{Z}^+$) and $\{0\}\cup\mathbb{Z}^+$ are equipped with their respective discrete $\sigma$-algebras; $\xi$ is the fixed number of subpopulations, called groups, in our population. Cartesian products of measurable spaces are equipped with their induced product $\sigma$-algebras. 

We introduce the following notation: For random elements $Y,Y'\in M$ and $S,S'\in\Xi$, we write $Y\mid S\overset{d}{=}Y'\mid S'$ if $Y\mid(S=s)$ and $Y'\mid(S'=s)$ are identically distributed for all $s$ in the supports of both $S$ and $S'$; that is, if $P(Y\leq y|S=s) = P(Y'\leq y|S'=s)$ for each $s\in\Xi$ such that
$P(S=s)P(S'=s)>0$. This relationship is not transitive in general, since the supports of both $S$ and $S'$ are involved.

\sloppy Consider a random element $(Y,S)\in M\times\Xi$, and a collection of random elements $\{(Y_i,S_i)\}_{i\in\mathbb{Z}^+}\subset M\times\Xi$. Define  $N^s_n:=\sum_{i=1}^{n}I\{S_i=s\}$ for all $n\in\mathbb{Z}^+,s\in\Xi$. Here, the $Y$ and $Y_i$ terms represent the variable of interest, while the $S$ and $S_i$ terms represent group membership; $N_n^s$ is the size of the $s$th sample, while $n$ is the combined size of all samples.

\begin{assumption}\label{assumption}
(a) $\{(Y_i,S_i)\}_{i\in\mathbb{Z}^+}$ are independent, (b) $Y_i\mid S_i\overset{d}{=}Y\mid S$ for all $i\in\mathbb{Z}^+$, (c) $\lim_{n\rightarrow\infty}N^s_n=\infty$ a.s. for all $s\in\Xi$, and (d) $P(S=s)>0$ for all $s\in\Xi$.
\end{assumption}

The above is assumed throughout the paper. $(Y,S)$ represents a random draw from our population of interest. If $\{(Y_i,S_i)\}_{i\in\mathbb{Z}^+}$ is iid from this population and Assumption \ref{assumption}(d) holds, then so do (a), (b) and (c); this corresponds to scenario 2 from the introduction. However, $\{(Y_i,S_i)\}_{i\in\mathbb{Z}^+}$ need not be iid, as (b) requires only units from the same group be identically distributed. Thus the $S_i$ variables may be deterministic, corresponding to scenario 1, or random but not iid, as in scenario 3, or a mixture of deterministic and random, as in scenario 4.

%Consider again the three scenarios in the introduction. It is clear how this framework covers scenario 3. Scenario 2 is covered by letting $n=t$, even while the $S_i$'s are random. Secnario 1 is covered by letting $n=t$, and the $S_i$'s follow a pre-determined pattern e.g. if we want the number of men and women in our sample to be equal, we could let $S_i=$male for odd $i$ and female for even $i$.

For any $s\in\Xi$ and $j\in\mathbb{Z}^+$, define $\Omega^s=\{\lim_{n\rightarrow\infty}N^s_n=\infty\}$ and the stopping times
    \begin{equation*}
    K_j^s=\begin{cases}
        \arg\min_{k\in\mathbb{Z}^+}\{\sum_{i=1}^kI\{S_i=s\}=j\} &~~\mbox{on $\Omega^s$} \\
        1 &~~\mbox{on $(\Omega^s)^c$,}
    \end{cases}
    \end{equation*}
    so that on $\Omega^s$, $K_j^s$ is the $j$th positive integer for which $S_i=s$. Then $Y_{(j)}^s:=Y_{K_j^s}$ and $S_{(j)}^s:=S_{K_j^s}$ are the variable of interest and group membership for the $j$th sampled unit from group $s$.
    
    \begin{proposition} \label{iid}
    (a) For all $s\in\Xi$ and $j\in\mathbb{Z}^+$, $K_j^s$ is measurable, as are $Y_{(j)}^s$ and $S_{(j)}^s$.
    
    (b) For all $s\in\Xi$ and $n\in\mathbb{Z}^+$, $\{N_n^s,Y^s_{(1)},Y^s_{(2)},\ldots\}$ are independent and satisfy $Y_{(j)}^s\overset{d}=Y \mid (S=s)$ for all $j\in\mathbb{Z}^+$.
        %The random elements $y_{(j)}^s \mid \Omega^s$ and $y_0 \mid (s_0=s)$ are identically distributed, while $y_{(1)}^s \mid \Omega^s,y_{(2)}^s \mid \Omega^s,\ldots$ are independent. 

    (c) For all $m_1,\ldots,m_\xi\in\mathbb{Z}^+$, the vectors $\{(Y^s_{(1)},\ldots,Y^s_{(m_s)})\}_{s\in\Xi}$ are independent.
    \end{proposition}

    Proposition \ref{iid}(c) ensures that samples from different groups are independent, while (b) ensures that $\{Y_{(j)}^s\}_{j\in\mathbb{Z}^s}$ is iid from the $s$th group; even more, it ensures that the size of the $s$th sample is independent of the sample itself. This is useful in the context of Proposition \ref{conv}.

    Even if a statistic based on $Y_{(1)}^s,\ldots,Y_{(m)}^s$ converges as $m\rightarrow\infty$, does the equivalent statistic based on $Y_{(1)}^s,\ldots,Y_{(N^s_n)}^s$ converge as $n\rightarrow\infty$, now that the sample size is random?
    %Given a convergent sequence of random variables $\{X_m\}_{m\in\mathbb{Z}^+}$ and random indices $\{N_n\}_{n\in\mathbb{Z}^+}$ that converge to infinity, we might expect $\lim_{n\rightarrow\infty}X_{N_n}$ to equal $\lim_{m\rightarrow\infty}X_m$. 
    As can be seen in Proposition \ref{conv}(b), this is the case for almost sure convergence, but extra conditions are required for convergence in distribution and even in probability, the simplest being independence between the variables and indices; otherwise, more complicated conditions such as Anscombe's condition are necessary---see \cite{Anscombe1952}. Thanks to Proposition \ref{iid}(b), these are not needed for our case.
    
    Specific cases of Proposition \ref{conv} exist in the literature; see for example Theorem 2.2 in \cite{Gut2012} or Theorem 1 in \cite{Galambos1992}. However, we have been unable to find a rigorous reference at the stated level of generality, so we include this proposition, which does not depend in any way on Assumption \ref{assumption}, for completion.

    \begin{proposition}\label{conv}
        Let $\{X_m\}_{m\in\mathbb{Z}^+}\subset M$ and $\{N_n\}_{n\in\mathbb{Z}^+}\subset\mathbb{Z}^+$ be collections of random elements.

        (a) The element $X_{N_n}$ is measurable.

        Now let $M$ be a complete, separable metric space with metric $d$, and $X$ an $M$-valued random element.

        (b) If $\lim_{n\rightarrow\infty}N_n\overset{a.s.}{=}\infty$ and $\lim_{m\rightarrow\infty}X_m\overset{a.s.}{=}X$, then $\lim_{n\rightarrow\infty}X_{N_n}\overset{a.s.}{=}X$.

        (c) Suppose $N_n$ and $(X_1,\ldots,X_m)$ are independent for any $n,m\in\mathbb{Z}^+$. If $\lim_{n\rightarrow\infty}N_n\overset{p}{=}\infty$ and $\lim_{m\rightarrow\infty}X_m\overset{p}{=}X$, then $\lim_{n\rightarrow\infty}X_{N_n}\overset{p}{=}X$.

        (d) Suppose $N_n$ and $X_m$ are independent for any $n,m\in\mathbb{Z}^+$. If $\lim_{n\rightarrow\infty}N_n\overset{p}{=}\infty$ and $X_m\rightsquigarrow X$ as $m\rightarrow\infty$, then $X_{N_n}\rightsquigarrow X$ as $n\rightarrow\infty$.
    \end{proposition}
Actually, the proof of Proposition \ref{conv}(d) holds even without separability or completeness.

    \section{Application to stratified sampling}\label{stratified}

Suppose $M$ is Euclidean with the standard Euclidean norm denoted by $\lVert\cdot\rVert$, and define an estimator for $\mu^s:=E(Y\mid S=s)$ by $\hat\mu^s_m:=(1/m)\sum_{j=1}^{m}Y_{(j)}^s$, where $m\in\mathbb{Z}^+$. Let $\hat\lambda_n^s\in[0,1]$ ($s\in \Xi$) be random variables satisfying $\sum_{s=1}^\xi\hat\lambda_n^s=1$. Then the stratified sampling mean estimator is $\hat\mu_n:=\sum_{s=1}^\xi\hat\lambda^s_n\hat\mu_{N_n^s}^s$.

\begin{proposition} \label{cons}
    Suppose $d\in\mathbb{Z}^+$, $M=\mathbb{R}^d$ and $E(\lVert Y\rVert\mid S=s)<\infty$ for all $s\in\Xi$. Assume that the random variables $\hat\lambda_n^s$ satisfy $\lim_{n\rightarrow\infty}\hat\lambda_n^s=P(S=s)$ almost surely/in probability for all $s\in\Xi$. Then $\lim_{n\rightarrow\infty}\hat\mu_n=E(Y)$ by the same mode of convergence.

    \end{proposition}

    \begin{proposition}\label{clt}
        Suppose $d\in\mathbb{Z}^+$, $M=\mathbb{R}^d$ and $E(\lVert Y\rVert^2\mid S=s)<\infty$ for all $s\in\Xi$. Denote the covariance matrices by $\Sigma^s:=\text{Cov}(Y\mid S=s)$, for all $s\in \Xi$. Assume that $\lambda_n^s=P(S=s)+o_p(n^{-1/2})$ as $n\rightarrow\infty$, and that $\lim_{n\rightarrow\infty}N^s_n/n=f_s$ for some $f_s\in(0,1)$, for all $s\in\Xi$. 

        (a) As $n\rightarrow\infty$,
        \begin{align*}
        n^{1/2}(\hat\mu_n-E(Y))\rightsquigarrow N\bigg(0,\sum_{s=1}^\xi \frac{P(S=s)^2}{f_s}\Sigma^s\bigg).
        \end{align*}

        (b) As $n\rightarrow\infty$
        \begin{align*}
            P\bigg(n\big(\hat\mu_n-E(Y)\big)^T\Big(\sum_{s=1}^\xi\frac{n(\hat\lambda_n^s)^2}{N^s_n}\hat\Sigma^s_n\Big)^{-1}\big(\hat\mu_n-E(Y)\big)<\chi^2_{1-\alpha,d}\bigg)\rightarrow1-\alpha,
        \end{align*}
        where $\chi^2_{1-\alpha,d}$ is the $(1-\alpha)$-quantile of the $\chi^2_d$-distribution, and $\hat\Sigma^s_n:=\frac{1}{N^s_n-1}\sum_{j=1}^{N^s_n}(Y_{(j)}^s-\hat\mu^s_{N^s_n})(Y_{(j)}^s-\hat\mu^s_{N^s_n})^T$.
    \end{proposition}

    Proposition \ref{clt}(a) is not actually obvious since $\{N^s_n\}_{s\in\Xi}$, and therefore $\{\hat\mu_{N_n^s}^s\}_{s\in\Xi}$, may not be independent.

    The proportion estimator $N_n^s/n$ would not work as $\hat\lambda_n^s$ in Proposition \ref{clt}. Here we provide two estimators that do:
    \begin{enumerate}
        \item $\hat\lambda_n^s=P(S=s)$, since $P(S=s) = P(S = s) + o_p(n^{-1/2})$ is trivially true. This is commonly done in stratified sampling, where the strata are weighted by population proportions that are presumed to be known beforehand.
        \item Let $\hat\lambda_n^s$ be the ordinary empirical proportion estimator of $P(S=s)$ based on an \textit{auxiliary} iid sample of size $m_n$, with $m_n/n\rightarrow\infty$, which only contains information about $S$ and not $Y$. Then $\hat\lambda_n^s-P(S=s)=O_p(m_n^{-1/2})=o_p(n^{-1/2})$. This requires no assumption of independence between the two samples. This kind of set up is common, as often information on $S$ (e.g. demographics) is much easier and cheaper to collect than information on $Y$.
    \end{enumerate}

    \section{Discussion}

    We have proposed a framework for asymptotic theory with group membership modeled as random variables. This provides greater flexibility by allowing for different kinds of deterministic and random membership, as well as mixtures of both. An important note is that although there are some surface-level similarities between finite mixture models and our framework in that the marginal distribution of $Y_i$ is expressible as a finite mixture, we go far beyond classical finite mixture models by dropping the iid assumption. Additionally, in finite mixture modeling, the $S_i$'s are unobserved and the goal is to infer their distribution from the $Y_i$'s; this is completely different from the goals of the current paper.

    We hope that our results are useful for others trying to prove asymptotic properties for statistics based on multiple samples with greater generality.

    \section*{Acknowledgement}

    This work was supported by the Soongsil University Research Fund (New Faculty Research Support, 2025).

    \appendix
    \section{Proofs}\label{proofs}

    \begin{proof}[Proof of Proposition \ref{iid}]
        (a) Fix $s\in\Xi$. For $K_j^s$, it suffices to show that $\{K_j^s=t\}$ is measurable for each $t\in\mathbb{Z}^+$. Noting that $\{K_j^s=t\}=\{\sum_{i=1}^{t-1}I\{S_i=s\}=j-1\}\cap\{S_t=s\}\cap\Omega^s\in\mathcal{F}$ for $t>1$ and $\{K_1^s=1\}=\{S_1=s\}\cup(\Omega^s)^c\in\mathcal{F}$ completes the proof. Then for all $t\in\mathbb{Z}^+$, $Y_{(j)}^s=Y_t$ on $\{K_j^s=t\}$, so $Y_{(j)}^s$ is measurable since $\{K_j^s=t\}$ is measurable for all $t\in\mathbb{Z}^+$. Similarly, $S_{K_j^s}$ is measurable. This result can also be derived from basic properties of stopping times; see for example Proposition 1.4 in Chapter 2 of \cite{Ethier1986}
    
    (b) Fix $s\in\Xi$. If $P(\Omega_*)=1$ for some $\Omega_*\in\mathcal{F}$, $P(A \mid \Omega_*)=P(A)$ for any set $A\in\mathcal{F}$, and sets (and random elements) are independent if and only if they are independent given $\Omega_*$.

For $t,t'\in\{0\}\cup\mathbb{Z}^+$, let $\Delta^s(t,t')$ be the random variable defined as the number of times $S_i=s$ after $i=t$ and before $i=t'$; that is, $\Delta^s(t,t')=0$ if $t'\leq t+1$ and $\sum_{i=t+1}^{t'-1}I\{S_i=s\}$ if $t'\geq t+2$.

\sloppy Then for any $m\in\mathbb{Z}^+$, any $j_1,\ldots,j_m\in\mathbb{Z}^+$ such that $j_1<\cdots<j_m$, and any $t_1,\ldots,t_m\in\mathbb{Z}^+$ such that $t_0:=0<t_1<\cdots<t_m$,
\begin{align}\label{sets}
\Big(\bigcap_{l=1}^m\{K_{j_l}^s=t_l\}\Big)\cap\Omega^s&=\Big(\bigcap_{l=1}^m\{S_{t_l}=s,\Delta^s(t_{l-1},t_l)=j_l-j_{l-1}-1\}\Big)\cap\Omega^s,
\end{align}
and $\Delta^s(0,t_1),(Y_{t_1},S_{t_1}),\Delta^s(t_1,t_2),(Y_{t_2},S_{t_2}),\ldots,\Delta^s(t_{j_{m-1}},t_{j_m}),(Y_{t_m},S_{t_m})$ are independent by Assumption \ref{assumption}(a). 

Letting $m=1$ and $j_1=j$ in the above paragraph, by Assumption \ref{assumption}(c), (b) and (d), we get $Y_{(j)}^s\overset{d}=Y \mid (S=s)$ from the following: for any $A$ in the Borel $\sigma$-algebra,
        \begin{equation}\begin{aligned}\label{id}
            &P(Y_{(j)}^s\in A)=P(Y_{(j)}^s\in A \mid \Omega^s)=\sum_{t=1}^\infty P(Y_t\in A, K_j^s=t \mid \Omega^s) \\
            =&\sum_{t=1}^\infty P(Y_t\in A,S_t=s,\Delta^s(0,t)=j-1 \mid \Omega^s) \\
            =&\sum_{t=1}^\infty P(Y_t\in A, S_t=s\mid\Omega^s)P(\Delta^s(0,t)=j-1 \mid \Omega^s) \\
            =&\sum_{t=1}^\infty P(Y_t\in A\mid S_t=s,\Omega^s)P(S_t=s,\Delta^s(0,t)=j-1 \mid \Omega^s) \\
            %&=\sum_{t=1}^\infty P(Y_t\in A \mid S_t=s)P(S_t=s \mid \Omega^s)P(\Delta^s(0,t)=j-1 \mid \Omega^s) \\
            %%%&=\sum_{t=1}^\infty P(Y\in A \mid S=s)P(\Delta^s(0,t)=j-1,S_t=s \mid \Omega^s) \\
            =&P(Y\in A \mid S=s)\sum_{t=1}^\infty P(K_j^s=t)=P(Y\in A \mid S=s).
        \end{aligned}\end{equation}
        proving identity of distribution. For independence, let $m\in\mathbb{Z}^+$ be arbitrary and $(j_1,\ldots,j_m)=(1,\ldots,m)$. Consider two cases: $k<m$ and $k\geq m$. By Assumption \ref{assumption}(c), (b), (a) and (d), for all positive integers $k<m$ and $A_l\in\mathcal{B}$ ($l=1,\ldots,m$),
        \begin{align*}
            &P\bigg(\{N^s_n=k\}\cap\bigcap_{l=1}^m\{Y_{(l)}^s\in A_l\}\bigg)=\sum_{t_1<\cdots t_k\leq n<t_{k+1}<\cdots<t_m}P\bigg(\bigcap_{l=1}^m\{Y_{t_l}\in A_l\}\cap \bigcap_{l=1}^m\{K_{l}^s=t_l\}\,\,\bigg|\,\,\Omega^s\bigg) \\
            %=&\sum_{t_1<\cdots t_k\leq n<t_{k+1}<\cdots<t_m}P\bigg(\bigcap_{l=1}^m\{Y_{t_l}\in A_l\}\cap\bigcap_{l=1}^m\{S_{t_l}=s,\Delta^s(t_{l-1},t_l)=0\}\,\,\bigg|\,\,\Omega^s\bigg) \\
            =&\sum_{t_1<\cdots t_k\leq n<t_{k+1}<\cdots<t_m}\bigg(\prod_{l=1}^mP(Y_{t_l}\in A_l,S_{t_l}=s \mid \Omega^s)\bigg)P\bigg(\bigcap_{l=1}^m\{\Delta^s(t_{l-1},t_l)=0\}\,\,\bigg|\,\,\Omega^s\bigg) \\
            %%%=&\sum_{t_1<\cdots t_k\leq n<t_{k+1}<\cdots<t_m}\bigg(\prod_{l=1}^mP(Y_{t_l}\in A_l \mid S_{t_l}=s,\Omega^s)P(S_{t_l}=s \mid \Omega^s)\bigg)P\bigg(\bigcap-{l=1}^m\{\Delta^s(t_{l-1},t_l)=0\}\,\,\bigg|\,\,\Omega^s\bigg) \\
            %=&\sum_{t_1<\cdots t_k\leq n<t_{k+1}<\cdots<t_m}P(Y_{t_1}\in A_1 \mid S_{t_1}=s)P(S_{t_1}=s)\ldotsP(Y_{t_m}\in A_m \mid S_{t_m}=s)P(S_{t_m}=s)P(\Delta^s(0,t_1)=\ldots=\Delta^s(t_{j_{m-1}},t_{j_m})=0 \mid \Omega^s) \\
            %=&\sum_{t_1<\cdots t_k\leq n<t_{k+1}<\cdots<t_m}P(Y_{t_1}\in A_1 \mid S_{t_1}=s)\ldotsP(Y_{t_m}\in A_m \mid S_{t_m}=s)P(S_{t_1}=s \mid \Omega^s)\ldotsP(S_{t_m}=s \mid \Omega^s)P(\Delta^s(0,t_1)=\ldots=\Delta^s(t_{j_{m-1}},t_{j_m})=0 \mid \Omega^s) \\
            =&\sum_{t_1<\cdots t_k\leq n<t_{k+1}<\cdots<t_m} \bigg(\prod_{l=1}^mP(Y_{t_l}\in A_l \mid S_{t_l}=s,\Omega^s)\bigg)P\bigg(\bigcap_{l=1}^m\{S_{t_l}=s,\Delta^s(t_{l-1},t_l)=0\}\,\,\bigg|\,\,\Omega^s\bigg) \\
            =&\bigg(\prod_{l=1}^mP(Y\in A_l \mid S=s)\bigg)\sum_{t_1<\cdots t_k\leq n<t_{k+1}<\cdots<t_m}P(K_{1}^s=t_1,\ldots,K_{m}^s=t_m) \\
            =&\bigg(\prod_{l=1}^mP(Y_{(l)}\in A_l)\bigg)P(N^s_n=k).
        \end{align*}
        where the last equality follows from (\ref{id}). Similarly, if $k\geq m$,
        \begin{align*}
            &P\bigg(\{N^s_n=k\}\cap\bigcap_{l=1}^m\{Y_{(l)}^s\in A_l\}\bigg)=\sum_{t_1<\cdots<t_k\leq n<t_{k+1}}P\bigg(\bigg[\bigcap_{l=1}^m\{Y_{t_l}\in A_l\}\bigg]\cap\bigg[\bigcap_{l={1}}^{k+1}\{K_l^s=t_l\}\bigg]\,\,\bigg|\,\,\Omega^s\bigg) \\
            %=&\sum_{t_1<\cdots<t_k\leq n<t_{k+1}}P\bigg(\bigg[\bigcap_{l=1}^m\{Y_{t_l}\in A_l\}\bigg]\cap\bigg[\bigcap_{l={1}}^{k+1}\{S_{t_l}=s,\Delta^s(t_{l-1},t_l)=0\}\bigg]\,\,\bigg|\,\,\Omega^s\bigg) \\
            =&\sum_{t_1<\cdots<t_k\leq n<t_{k+1}}\bigg(\prod_{l=1}^mP(Y_{t_l}\in A_l,S_{t_l}=s \mid \Omega^s)\bigg)P\bigg(\bigcap_{l=1}^m\{\Delta^s(t_{l-1},t_l)=0\} \\
            &\quad\quad\cap\bigcap_{l={m+1}}^{k+1}\{S_{t_l}=s,\Delta^s(t_{l-1},t_l)=0\}\,\,\bigg|\,\,\Omega^s\bigg) \\
            %%%=&\sum_{t_1<\cdots<t_k\leq n<t_{k+1}}\bigg(\prod_{l=1}^mP(Y_{t_l}\in A_l \mid S_{t_l}=s,\Omega^s)P(S_{t_l}=s \mid \Omega^s)\bigg)P\bigg(\bigcap-{l=1}^m\{\Delta^s(t_{l-1},t_l)=0\}\,\,\bigg|\,\,\Omega^s\bigg) \\
            %=&\sum_{t_1<\cdots<t_k\leq n<t_{k+1}}P(Y_{t_1}\in A_1 \mid S_{t_1}=s)P(S_{t_1}=s)\ldotsP(Y_{t_m}\in A_m \mid S_{t_m}=s)P(S_{t_m}=s)P(\Delta^s(0,t_1)=\ldots=\Delta^s(t_{j_{m-1}},t_{j_m})=0 \mid \Omega^s) \\
            %=&\sum_{t_1<\cdots<t_k\leq n<t_{k+1}}P(Y_{t_1}\in A_1 \mid S_{t_1}=s)\ldotsP(Y_{t_m}\in A_m \mid S_{t_m}=s)P(S_{t_1}=s \mid \Omega^s)\ldotsP(S_{t_m}=s \mid \Omega^s)P(\Delta^s(0,t_1)=\ldots=\Delta^s(t_{j_{m-1}},t_{j_m})=0 \mid \Omega^s) \\
            =&\sum_{t_1<\cdots<t_k\leq n<t_{k+1}}\bigg(\prod_{l=1}^mP(Y_{t_l}\in A_l \mid S_{t_l}=s,\Omega^s)\bigg)P\bigg(\bigcap_{l=1}^{k+1}\{S_{t_l}=s,\Delta^s(t_{l-1},t_l)=0\}\,\,\bigg|\,\,\Omega^s\bigg)\\
            =&\Big(\prod_{l=1}^mP(Y\in A_l \mid S=s)\Big)\sum_{t_1<\cdots<t_k\leq n<t_{k+1}}P(K_{1}^s=t_1,\ldots,K_{k+1}^s=t_{k+1}) \\
            =&P(N^s_n=k)\prod_{l=1}^mP(Y\in A_l \mid S=s).
        \end{align*}

        (c) In the following,  each $A_l^s\in\mathcal{B}$ and the sums are taken over all values of $t_1^1,\ldots,t_{m_1}^1,\ldots,t_1^\xi,\ldots,t_{m_s}^\xi\in\mathbb{Z}^+$ for which $t_1^s<\ldots<t_{m_s}^s$ for each $s\in\Xi$ and $t_l^s\neq t_{l'}^{s'}$ if $(s,l)\neq(s',l')$. The last equality comes from the results in (b).
        \begin{align*}
            &P\bigg(\bigcap_{s=1}^\xi\bigcap_{l=1}^{m_s}\{Y^s_{(j)}\in A_l^s\}\bigg)=\sum P\bigg(\bigg[\bigcap_{s=1}^\xi\bigcap_{l=1}^{m_s}\{Y_{t_l^s}\in A_l^s\}\bigg]\cap\bigg[\bigcap_{s=1}^\xi\bigcap_{l=1}^{m_s}\{K_l^s=t_l^s\}\bigg]\,\,\bigg|\,\,\Omega^s\bigg) \\
            %=&\sum P\bigg(\bigg[\bigcap_{s=1}^\xi\bigcap_{l=1}^{m_s}\{Y_{t_l^s}\in A_l^s\}\bigg]\cap\bigg[\bigcap_{s=1}^\xi\bigcap_{l=1}^{m_s}\{S_{t_l^s}=s,\Delta^s(t_{l-1}^s,t_l^s)=0\}\bigg]\,\,\bigg|\,\,\Omega^s\bigg) \\
            =&\sum \bigg(\prod_{s=1}^\xi\prod_{l=1}^{m_s}P(Y_{t_l^s}\in A_l^s,S_{t_l^s}=s\mid \Omega^s)\bigg)P\bigg(\bigcap_{s=1}^\xi\bigcap_{l=1}^{m_s}\{\Delta^s(t_{l-1}^s,t_l^s)=0\}\,\,\bigg|\,\,\Omega^s\bigg) \\
            =&\sum \bigg(\prod_{s=1}^\xi\prod_{l=1}^{m_s}P(Y_{t_l^s}\in A_l^s\mid S_{t_l^s}=s,\Omega^s)\bigg)P\bigg(\bigcap_{s=1}^\xi\bigcap_{l=1}^{m_s}\{S_{t_l^s}=s,\Delta^s(t_{l-1}^s,t_l^s)=0\}\,\,\bigg|\,\,\Omega^s\bigg) \\
            =&\bigg(\prod_{s=1}^\xi\prod_{l=1}^{m_s}P(Y\in A_l^s\mid S=s)\bigg)\sum P\bigg(\bigcap_{s=1}^\xi\bigcap_{l=1}^{m_s}\{K_l^s=t_l^s\}\bigg)=\prod_{s=1}^\xi P(Y^s_{(1)}\in A_1^s,\ldots,Y^s_{(m_s)}\in A_{m_s}^s).
        \end{align*}
    \end{proof}

    \begin{proof}[Proof of Proposition \ref{conv}]
        (a) For any $A\in\mathcal{B}$, $\{X_{N_n}\in A\}=\cup_{m=1}^\infty(\{X_m\in A\}\cap\{N_n=m\})$; as a countable union of measurable sets, this is measurable.

        \sloppy (b) Let $\Omega'\subset\Omega$ be the subset on which $\lim_{m\rightarrow\infty}X_m=X(\omega)$ and $\Omega''\subset\Omega$ the subset on which $\lim_{n\rightarrow\infty}N_n(\omega)=\infty$. For any given $\omega\in\Omega''\cap\Omega'$, let $m_n:=N_n(\omega)$ for all $n\in\mathbb{Z}^+$. Then $\lim_{n\rightarrow\infty}m_n=\infty$ implies $\lim_{n\rightarrow\infty}X_{N_n}(\omega)=\lim_{n\rightarrow\infty}X_{m_n}(\omega)=\lim_{m\rightarrow\infty}X_{m}(\omega)=X(\omega)$; a.s. convergence follows since $P(\Omega''\cap\Omega')=1$.

        (c) For any $\epsilon>0$ and $\eta>0$, there exists some $m_0\in\mathbb{Z}^+$ for which $P(d(X,X_m)>\epsilon/3)<\eta/3$ if $m\geq m_0$. This also implies that $P(d(X_{m_0},X_m)>2\epsilon/3)<P(\{d(X_{m_0},X)>\epsilon/3\}\cup\{d(X,X_m)>\epsilon/3\})<2\eta/3$ for all $m\geq m_0$. There exists some $n_0\in\mathbb{Z}^+$ for which $P(1/N_n> 1/m_0)=P(N_n< m_0)<\eta/3$ if $n\geq n_0$. Therefore if $n\geq n_0$,
        \begin{align*}
            P(d(X,X_{N_n})>\epsilon)\leq&P\bigg(\bigg\{d(X,X_{m_0})>\frac{\epsilon}{3}\bigg\}\cup\bigg\{d(X_{m_0},X_{N_n})>\frac{2\epsilon}{3}\bigg\}\bigg) \\
            \leq&\frac{\eta}{3}+\sum_{m=m_0}^\infty P\bigg(d(X_{m_0},X_m)>\frac{2\epsilon}{3},N_n=m\bigg)\leq\frac{\eta}{3}+\frac{2\eta}{3}\sum_{m=m_0}^\infty P(N_n=m)\leq \eta.
        \end{align*}

        (d) The proof is similar. Let $A$ be a continuity set of the distribution of $X$. Then for any $\eta>0$ there exists some $m_0$ for which $\lvert P(X\in A)-P(X_m\in A)\rvert<\eta/5$ if $m\geq m_0$. This also implies that $\lvert P(X_{m_0}\in A)-P(X_m\in A)\rvert\leq \lvert P(X_{m_0}\in A)-P(X\in A)\rvert+\lvert P(X\in A)-P(X_m\in A)\rvert <2\eta/5$ for all $m\geq m_0$. There exists some $n_0\in\mathbb{Z}^+$ for which $P(N_n< m_0)<\eta/5$ if $n\geq n_0$. Therefore if $n\geq n_0$,
        \begin{align*}
            &\lvert P(X\in A)-P(X_{N_n}\in A)\rvert\leq\lvert P(X\in A)-P(X_{m_0}\in A)\rvert+\lvert P(X_{m_0}\in A)-P(X_{N_n}\in A)\rvert \\
            <&\frac{\eta}{5}+\Big\lvert P(X_{m_0}\in A,N_n<m_0)-P(X_{N_n}\in A,N_n<m_0) \\
            &\quad+\sum_{m=m_0}^\infty P(X_{m_0}\in A,N_n=m)-P(X_m\in A,N_n=m)\Big\rvert\\
            <&\frac{\eta}{5}+\frac{\eta}{5}+\frac{\eta}{5}+\sum_{m=m_0}^\infty\lvert P(X_{m_0}\in A)-P(X_m\in A)\rvert P(N_n=m)<\frac{3\eta}{5}+\frac{2\eta}{5}\sum_{m=m_0}^\infty P(N_n=m)\leq \eta.
        \end{align*}
    \end{proof}

    \begin{proof}[Proof of Proposition \ref{cons}]
    Applying Propositions \ref{iid}, \ref{conv}(b) and (c), and the strong law of large numbers, $\lim_{n\rightarrow\infty}(1/N^s_n)\sum_{j=1}^{N^s_n}Y_{(j)}^s\overset{a.s.}{=} E(Y \mid S=s)$. Since a.s. convergence is preserved by products and sums, $\lim_{n\rightarrow\infty}\hat\mu_n=\sum_{s=1}^\xi P(S=s)E(Y\mid S=s)=E(E(Y\mid S))=E(Y)$ by the same mode of convergence by which each $\hat\lambda_n^s$ converges.
    \end{proof}

    \begin{proof}[Proof of Proposition \ref{clt}]

(a) Fix $s\in\Xi$. Define $a(m,n):=\min\{m,\lfloor nf_s\rfloor\}$ and $b(m,n):=\max\{m,\lfloor nf_s\rfloor\}$ for all $m,n\in\mathbb{Z}^+$. Then define
\begin{align*}
    X_{m,n}:=&m^{1/2}(\hat\mu_{m}^s-\mu^s)-\lfloor nf_s\rfloor^{1/2}(\hat\mu_{\lfloor nf_s\rfloor}^s-\mu^s) \\
    =&\bigg(\frac{1}{a^{1/2}(m,n)}-\frac{1}{b^{1/2}(m,n)}\bigg)\sum_{j=1}^{a(m,n)}(Y_{(j)}^s-\mu^s)+\frac{1}{b^{1/2}(m,n)}\sum_{j=a(m,n)+1}^{b(m,n)}(Y_{(j)}^s-\mu^s),
\end{align*}
where $m$ may be fixed or random. Then thanks to Proposition \ref{iid}(b), $E(X_{N^s_n,n}\mid N^s_n=m)=E(X_{m,n}\mid N^s_n=m)=E(X_{m,n})=0$, so $E(X_{N^s_n,n})=E[E(X_{N^s_n,n}\mid N^s_n)]=0$. Therefore
\begin{align*}
    &E(X_{N^s_n,n}X_{N^s_n,n}^T)=\text{Cov}(X_{N^s_n,n}) \\
    =&E[\text{Cov}(X_{N^s_n,m}\mid N^s_n)]+\text{Cov}[E(X_{N^s_n,m}\mid N^s_n)]=E[\text{Cov}(X_{N^s_n,m}\mid N^s_n)] \\
    =&\sum_{m=0}^\infty P(N^s_n=m)\bigg[\bigg(\frac{1}{a^{1/2}(m,n)}-\frac{1}{b^{1/2}(m,n)}\bigg)^2a(m,n)+\frac{1}{b(m,n)}(b(m,n)-a(m,n))\bigg]\Sigma^s \\
    =&\sum_{m=0}^\infty P(N^s_n=m)\bigg(2-2\frac{a^{1/2}(m,n)}{b^{1/2}(m,n)}\bigg)\Sigma^s.
\end{align*}
Because $\lim_{n\rightarrow\infty}N^s_n/\lfloor nf_s\rfloor\overset{p}{=}1$ implies $\lim_{n\rightarrow\infty}a^{1/2}(N^s_n,n)/b^{1/2}(N^s_n,n)\overset{p}{=}1$, for every $\epsilon>0$ there exists some $n_0\in\mathbb{Z}^+$ for which $n\geq n_0$ implies $P(1-a^{1/2}(N^s_n,n)/b^{1/2}(N^s_n,n)>\epsilon/4)<\epsilon/4$. Defining $B_n:=\{m\in\mathbb{Z}^+:1-a^{1/2}(m,n)/b^{1/2}(m,n)>\epsilon/4\}$ and $C_n:=\{m\in\mathbb{Z}^+:1-a^{1/2}(m,n)/b^{1/2}(m,n)\leq\epsilon/4\}$, if $n\geq n_0$,
\begin{align*}
    &\sum_{m\in B_n}P(N^s_n=m)\bigg(2-2\frac{a^{1/2}(m,n)}{b^{1/2}(m,n)}\bigg)+\sum_{m\in C_n}P(N^s_n=m)\bigg(2-2\frac{a^{1/2}(m,n)}{b^{1/2}(m,n)}\bigg) \\
    \leq&2P(1-a^{1/2}(N^s_n,n)/b^{1/2}(N^s_n,n)>\epsilon/4)+\sum_{m\in C_n}P(N^s_n=m)\frac{2\epsilon}{4}<\epsilon.
\end{align*}
Therefore $E(\lVert X_{N^s_n,n}\rVert^2)=\text{tr}(E(X_{N^s_n,n}X_{N^s_n,n}^T))\rightarrow 0$, and hence $X_{N^s_n,n}\overset{p}\rightarrow 0$, as $n\rightarrow\infty$.

    Define $\{Z_s\}_{s\in\Xi}$ to be independent with $Z_s\sim N(0,\Sigma^s)$ ($s\in \Xi)$. By Proposition \ref{iid}(c) and Proposition \ref{conv}(d), as $n\rightarrow\infty$,
    \begin{align*}
        &\big(\lfloor nf_1\rfloor^{1/2}[\hat\mu_{\lfloor nf_1\rfloor}^1-\mu^1],\ldots,\lfloor nf_\xi\rfloor^{1/2}[\hat\mu_{\lfloor nf_\xi\rfloor}^\xi-\mu^\xi]\big)\rightsquigarrow(Z_1,\ldots,Z_\xi),
    \end{align*}
    With the conclusion of the above paragraph and Slutsky's theorem, this implies
    \begin{align}\label{law}
        &\big((N^1_n)^{1/2}[\hat\mu_{N^1_n}^1-\mu^1],\ldots,(N^\xi_n)^{1/2}[\hat\mu_{N^\xi_n}^\xi-\mu^\xi]\big)\rightsquigarrow(Z_1,\ldots,Z_\xi),
    \end{align}
    
    By assumption, $\lim_{n\rightarrow\infty}(\hat\lambda^1_n(n/N^1_n)^{1/2},\ldots,\hat\lambda^\xi_n(nN^\xi_n)^{1/2})\overset{p}{=}(P(S=1)/(f_1)^{1/2},\ldots,P(S=\xi)/(f_\xi)^{1.2})$ as $n\rightarrow\infty$. With (\ref{law}) and Slutsky's theorem, this means
    \begin{align}\label{step1}
        n^{1/2}\bigg(\hat\mu_n-\sum_{s=1}^\xi\hat\lambda^ s_n\mu^s\bigg)\rightsquigarrow\sum_{s=1}^\xi\frac{P(S=1)}{(f_s)^{1/2}}Z_s\sim N\bigg(0,\sum_{s=1}^\xi \frac{P(S=s)^2}{f_s}\Sigma^s\bigg)
    \end{align}
    as $n\rightarrow\infty$. Since $E(Y)=E[E(Y\mid S)]=\sum_{s=1}^\xi P(S=s)\mu^s$, as $n\rightarrow\infty$,
    \begin{align}\label{step2}
        n^{1/2}\big(\hat\mu_n-E(Y)\big)-n^{1/2}\bigg(\hat\mu_n-\sum_{s=1}^\xi\hat\lambda^ s_n\mu^s\bigg)= n^{1/2}\sum_{s=1}^\xi(\hat\lambda^s_n-P(S=s))\mu^s\overset{p}{\rightarrow}0
    \end{align}
since $\hat\lambda^s_n-P(S=s)=o_p(1/\sqrt{n})$ as $n\rightarrow\infty$. The result follows from (\ref{step1}), (\ref{step2}) and Slutsky's theorem.
    
    (b) For all $s\in\Xi$, as $n\rightarrow\infty$, $\hat\lambda^s_n\rightarrow P(S=s)$, $n/N^s_n\rightarrow 1/f_s$, and by Proposition \ref{conv}(c), $\hat\Sigma^s_n\overset{p}{\rightarrow}\Sigma^s$; Proposition \ref{iid}(b), Assumption \ref{assumption} and the consistency of sample covariance matrices of iid data imply that the conditions of Proposition \ref{conv}(c) are met. Thus $\lim_{n\rightarrow\infty}\sum_{s=1}^\xi[n(\hat\lambda_n^s)^2/N^s_n]\hat\Sigma^s_n\overset{p}{=}\sum_{s=1}^\xi [P(S=s)^2/f_s]\Sigma^s$, and the result follows from (a).
    \end{proof}

 \bibliographystyle{apalike}
\bibliography{references}

\end{document}